\documentclass[11pt,a4paper]{article}
\usepackage{amsmath,amssymb,amsthm}
\usepackage{graphicx}

\usepackage{amsmath, amsthm, amsfonts, amssymb, multicol, pstricks}

\newtheorem{theorem}{Theorem}[section]

\newtheorem{lemma}[theorem]{Lemma}
\newtheorem{prop}[theorem]{Proposition}
\theoremstyle{definition}
\newtheorem{definition}[theorem]{Definition}

\newtheorem{remark}{Remark}[section]
\numberwithin{equation}{section}

\newcommand{\vp}{\varphi}
\newcommand{\om}{\omega}
\newcommand{\U}{\mathcal{U}}
\newcommand{\real}{\mathbb{R}}
\newcommand{\rn}{\real^{n}}

\let\leq\leqslant

\let\geq\geqslant

\numberwithin{equation}{section}
\allowdisplaybreaks

\begin{document}
	
	\title{Criterion for exponential dichotomy of periodic generalized linear differential equations and an application to admissibility.}
	\author{Claudio A. Gallegos\thanks{ Supported by FONDECYT Postdoctorado No 3220147. Universidad de Chile, Departamento de Matem\'aticas. Casilla 653, Santiago, CHILE. E-mail: {\tt  claudio.gallegos.castro@gmail.com}}, Gonzalo Robledo V.\thanks{Partially supported by FONDECYT Regular No 1210733 . Universidad de Chile, Departamento de Matem\'aticas. Casilla 653, Santiago, CHILE. E-mail: {\tt grobledo@uchile.cl}.}
	}
	
	\date{}
	\maketitle

	\begin{abstract}
		In this paper, we provide a necessary and sufficient condition ensuring the property of exponential dichotomy for periodic linear systems of generalized differential equations. This condition allow us to revisit a recent result of admissibility, obtaining an alternative formulation with a particular simplicity.
	\end{abstract}
	
	

	\pagestyle{myheadings} \markboth{\hfil C. A. Gallegos, G. Robledo V. \hfil $\hspace{3cm}$ } {\hfil$\hspace{1.5cm}$
		{Dichotomy for $\om$-periodic generalized linear differential equations}
		\hfil}

	\section{Introduction}
	The characterisation and study of the exponential dichotomy property for generalized linear differential equations (GLDEs) started recently in the seminal work of Bonotto, Federson and Santos \cite{BFS}. In that article, the property was introduced for Banach space valued functions and applied to obtain admissibility results for linear inhomogeneous systems having bounded or periodic components. Subsequently, a second step was made by the aforementioned authors in a series of roughness results for the property of dichotomy, see \cite{BFS2}. In addition,  in those articles the authors extend his results to measure differential equations and impulsive differential equations, which are well known particular cases of GLDEs.
	
	The aim of this brief article is to provide a necessary and sufficient condition ensuring the property of exponential dichotomy for $\om$--periodic homogeneous GLDEs. Our approach is performed by using the Floquet theory in the context of GLDEs, which was developed by Schwabik in \cite{SchwF} and \cite{SCHWABIK3}. Moreover, as an application of this characterisation, we provide an alternative perspective for a proposed admissibility result established in \cite{BFS}: An inhomogeneous $\om$--periodic GLDE, with homogeneous part having an exponential dichotomy, possesses a unique $\om$--periodic solution.
	
	The structure of this article is the following: Section 2 reviews fundamental facts and results of GLDEs. Section 3 states the above mentioned necessary and sufficient condition for the exponential dichotomy in the framework of $\om$--periodic GLDEs. Finally, Section 4 reexamines the admissibility result stated in \cite[Prop.4.8]{BFS}.

	
	\section{Notations and recalls}
	Let us commence by establishing the functional setting in which this article is inserted: A subset $P=\{\alpha_0,\alpha_1,...,\alpha_{\nu(P)}\}\subset [a,b]$, with $\nu(P)\in\mathbb{N}$, is said to be a \textit{partition} of $[a,b]\subset \real$ if $\alpha_0=a<\alpha_1<...<\alpha_{\nu(P)}=b$. We denote by  $\mathcal{P}([a,b])$ the set of all partitions of $[a,b]$.
	
	Let $(X,\|\cdot\|_{X})$ be an arbitrary Banach space and $\Phi:[a,b]\to X$ a function. We denote the \textit{variation} of $\Phi$ over the interval $[a,b]$ by
	\[
	\text{var}_{a}^{b}(\Phi):=\sup_{P\in \mathcal{P}([a,b])}\sum_{j=1}^{\nu(P)}\|\Phi(\alpha_{j})-\Phi(\alpha_{j-1})\|_{X}.
	\]
	The vector space consisting of all functions $\Phi\colon [a,b]\to X$ such that $\text{var}_{a}^{b}(\Phi)<\infty$ is denoted by $BV([a,b],X)$, and it is a Banach space with the norm
	\[
	\|\Phi\|_{BV}:=\|\Phi(a)\|_{X}+\text{var}_{a}^{b}(\Phi).
	\]
	We denote by $BV_{loc}(\real,X)$ the set consisting of all functions $\Phi\colon\real\to X$ of locally bounded variation, \emph{i.e.} all functions $\Phi$ such that $\text{var}_{a}^{b}(\Phi)<\infty$ for every compact interval $[a,b]\subset \real$.
	
	A function $\Phi : [a, b] \to  X$ is said to be
	{\it regulated} (see \cite[Section VII.6]{D}) if for every $t \in [a, b)$ the right-sided limit $\Phi(t^{+})$
	exists in $X$, and for every $t \in (a, b]$ the left-sided limit $\Phi(t^{-})$ exists in $X$. We denote by $G([a, b], X)$ the set consisting of all regulated functions from
	$[a, b]$ into $X$. The set  $G([a, b], X)$ endowed with the norm of uniform convergence is a Banach space. Moreover, every  $\Phi\in G([a, b], X)$ is the uniform limit of a sequence of step functions, and the set of discontinuities of $\Phi$ is at most countable. In addition, the following inclusion holds $BV([a,b],X)\subset G([a,b],X)$.  We denote by $G(\real,X)$ the space of all functions $\Phi:\real\to X$ such that the restriction of $\Phi$ to $[a,b]$ belongs to $G([a,b],X)$, for all $[a,b]\subset\real$. Moreover, throughout this paper we will use the following notations
	\[
	\Delta^{+}\Phi(t) := \Phi(t^{+}) - \Phi(t),\quad \text{ and } \quad \Delta^{-}\Phi(t) := \Phi(t ) - \Phi(t^{-}).
	\] 
	
	Given $\om>0$, we said that a function $\Phi\in G(\real,X)$ is $\om$-periodic if $\Phi(\om+t)-\Phi(t)=0$ for all $t\in\real$. We consider the vector normed space $G_{\om}(\real,X):=\{\Phi\in G(\real,X): \Phi\;\text{is $\om$-periodic}\}$, endowed with the norm $\displaystyle\|\Phi\|_{\om}:=\sup_{t\in[0,\om]}\|\Phi(t)\|$. Clearly, every function $\Phi$ in $G_{\om}(\real,X)$ is uniformly bounded on $\real$ and $(G_{\om}(\real,X),\|\cdot\|_{\om})$ is a Banach space.

	In the next, let us introduce notation in order to define the concept of integration which will be considered throughout this work. As usual, we denote by $\mathcal{L}(\rn)$ the vector space consisting of all $n\times n$-matrices endowed with the operator norm, and by $I$ the $n\times n$ - identity matrix.
	\begin{itemize}
		\item The finite collection of point-interval pairs  $(P,\tau):=\{(\tau_j,[\alpha_{j-1},\alpha_j]): j=1,...,\nu(P)\}$, where $P\in\mathcal{P}([a,b])$ and $\tau_j\in[\alpha_{j-1},\alpha_j]$, is called a \emph{tagged partition} of $[a,b]$.
		\item {Any positive function} $\delta:[a,b]\to\mathbb{R}^{+}$ is called a {\itshape gauge} on $[a,b]$.
		\item If $\delta$ is a gauge on $[a, b]$, a
		tagged partition $(P,\tau)$ is called  $\delta$-{\itshape fine} if
		\[
		[\alpha_{j-1},\alpha_j] \subset \left(\tau_j-\delta(\tau_j), \tau_j+\delta(\tau_j)\right), \; \; j=1, \ldots ,\nu(P).
		\]
		
	\end{itemize}

	\begin{definition}\label{Kint}
		A function $f\colon[a,b]\to\real^{n}$ is said to be Kurzweil--Stieltjes integrable with respect to the function $A \colon [a,b]\to\mathcal{L}(\rn)$ on $[a,b]$, if there is an element $ \mathcal{X} \in \rn$ having the following property:
		for every $\varepsilon>0$, there is a gauge $\delta(\cdot)$ on $[a,b]$ such that
		\[
		\left\|\sum_{j=1}^{\nu(P)}[A(\alpha_j)-A(\alpha_{j-1})]f(\tau_j)- \mathcal{X} \right\|<\varepsilon,
		\]
		for all $\delta$--fine tagged partition $(P,\tau)$ of $[a,b]$. In this case, $\mathcal{X}$ is called the \emph{Kurzweil--Stieltjes integral} of $f$ with respect to $A$ over $[a,b]$ and will be denoted by $\int_{a}^{b}{\rm d}[A(s)]f(s)$.
		
		If $\int_{a}^{b}{\rm d}[A(s)]f(s)$ exists then we define $\int_{b}^{a}{\rm d}[A(s)]f(s)=-\int_{a}^{b}{\rm d}[A(s)]f(s)$ and  for every $c \in[a,b]$ we set $\int_{c}^{c}{\rm d}[A(s)]f(s)=0$.
	\end{definition}
	
	Throughout this work, we will always considering this type of integral. 
	\begin{remark}
		We emphasize that the Kurzweil--Stieltjes integral $\int_{a}^{b}{\rm d}[A(s)]f(s)$ exists when at least one of the function involved is regulated and the other is of bounded variation. For a list of basic properties of this concept of integration, we refer to the reader the articles \cite{MT2,MT1,SCHWABIK} and the book \cite{MST}, along with the references there.  
	\end{remark}

	\subsection{Generalized linear differential equations}
	The Kurzweil--Stieltjes integral is an indispensable concept in the theory of GLDEs, which are essentially determined through an integral equation involving this notion of integration, see for instance \cite[Chapter~VI]{SCHWABIK1} or \cite[Chapter~7]{MST}. With no intention to do an exhaustive list of preliminary results, we will only mention here the most related to the subsequent sections. For a detailed summary of GLDEs, the reader can consult the books \cite{MST,SCHWABIK1,SCHWABIK3}.
	
	Throughout this article, we will always consider a function $A:\real\to\mathcal{L}(\rn)$ in $BV_{loc}(\real,\mathcal{L}(\rn))$, a function $f:\real\to\rn$ in $G(\real,\rn)$, and  we assume that the following condition holds:
	\begin{itemize}
		\item [{\bf (H)}] the matrices $[I-\Delta^{-}A(t)]$ and $[I+\Delta^{+}A(t)]$ are invertible for all $t\in\real$.
	\end{itemize} 
	
	Consider the GLDE of type
	\begin{equation}\label{nhGLDE}
	\dfrac{dx}{d\tau}=D[A(t)x+f(t)],
	\end{equation}
	and the homogeneous GLDE
	\begin{equation}\label{hGLDE}
	\dfrac{dx}{d\tau}=D[A(t)x].
	\end{equation}
	Similarly to the classical theory of nonautonomous linear systems, there exists a uniquely determined $n\times n$ - matrix valued function $U:\real\times\real\to\mathcal{L}(\rn)$, which is the transition matrix associated to the homogeneous GLDE \eqref{hGLDE} and satisfies
	\begin{equation}\label{U}
	U(t,s)=I+\int_{s}^{t}{\rm d}[A(r)]U(t,r), \; \; \text{ for all $t,s\in\real$.}
	\end{equation}
	Moreover, the transition matrix has the following properties:
	\begin{itemize}
		\item[{\rm (a)}] $U(t,t)=I$, for all $t\in \real$.
		\item[{\rm (b)}] For every compact interval $[a,b]\subset \real$, there exists a constant $M\geq0$ such that 
		\begin{multicols}{2}
			\begin{itemize}
				\item[{\rm(i)}] $\|U(t,s)\|\leq M$ for all $t,s\in[a,b]$,
				\item[{\rm(ii)}] {\rm var}$_{a}^{b}(U(t,\cdot))\leq M$ for $t\in[a,b]$,
				\item[{\rm(iii)}] {\rm var}$_{a}^{b}(U(\cdot,s))\leq M$ for $s\in[a,b]$,
			\end{itemize}
		\end{multicols}
		\item[{\rm (c)}] For every $r,s,t\in \real$ the relation $U(t,s)=U(t,r)U(r,s)$ holds.
		\item [{\rm (d)}] $U(t,s)\in\mathcal{L}(\rn)$ is invertible for every $s,t\in \real$, and the relation $[U(t,s)]^{-1}=U(s,t)$ holds.
		\item [{\rm (e)}] For $s,t\in \real$, the following identities hold
		\begin{multicols}{2}
			\begin{itemize}
				\item [{\rm(i)}] $U(t^{+},s)=[I+\Delta^{+}A(t)]U(t,s)$,
				\item [{\rm(ii)}] $U(t^{-},s)=[I-\Delta^{-}A(t)]U(t,s)$,
				\item [{\rm(iii)}] $U(t,s^{+})=U(t,s)[I+\Delta^{+}A(t)]^{-1}$,
				\item [{\rm(iv)}] $U(t,s^{-})=U(t,s)[I-\Delta^{-}A(t)]^{-1}$.
			\end{itemize}
		\end{multicols}
		
	\end{itemize}
	\begin{remark}
		The Eq. \eqref{U} should be understood in the sense that for every $k\in\{1,2,...,n\}$, the following equality holds
		\[
		U_{k}(t,s)=e_{k}+\int_{s}^{t}{\rm d}[A(r)]U_{k}(t,r), \; \; \text{ for all $t,s\in\real$},
		\]
		where $U_{k}$ denotes the k-th column of $U$, and $e_{k}$ is the k-th column of the identity matrix $I$.
	\end{remark}
	
	We finalize this subsection by recalling the variation of constants formula for GLDEs \eqref{nhGLDE}. 
	
	\begin{theorem}
		Assume that $A\in BV_{loc}(\real,\mathcal{L}(\rn))$, $f\in G(\real,\rn)$, and {\rm (H)} holds. Then for every $(x_0,s_0)\in \rn\times \real$, there exists a unique solution of \eqref{nhGLDE} with initial condition $x(s_0)=x_0$, that can be written in the form 
		\begin{equation}\label{vcf}
		x(t,s_0,x_0)=U(t,s_0)x_0 + f(t)-f(s_0)-\int_{s_0}^{t}{\rm d}[U(t,s)](f(s)-f(s_0)),
		\end{equation}
		for all $t\in \real$, where $U$ is the transition matrix defined by \eqref{U}.
	\end{theorem}

	\subsection{Exponential dichotomy}
	The property of exponential dichotomy for GLDEs was introduced in \cite{BFS} by following the lines of the classical framework and considering Banach space valued functions.  Throughout this article, we will restrict ourselves to the finite dimensional case.
	
	\begin{definition}{\rm (\cite[Def.3.1]{BFS})}
		The homogeneous GLDE \eqref{hGLDE} admits an {\it exponential dichotomy} on $\mathbb{R}$ if there exist positive constants $K,\alpha$ and a projection $P:\rn\to\rn$, i.e. $P^{2}=P$, such that 
		\begin{displaymath}
		\left\{\begin{array}{rcl}
		\|\U(t)P\U^{-1}(s)\| &\leq & Ke^{-\alpha(t-s)}, \quad \textnormal{for all $t,s\in\real$ with $t\geq s$},\\
		\|\U(t)(I-P)\U^{-1}(s)\|&\leq & Ke^{-\alpha(s-t)}, \quad \textnormal{for all $t,s\in\real$ with $s\geq t$}.
		\end{array}\right.		
		\end{displaymath}
		Where the function $\U:\real\to\mathcal{L}(\rn)$ is defined by $\U(t)=U(t,0)$ for all $t\in\real$.
	\end{definition}
	
	To simplify the redaction of the future statements, we will consider the following condition:
	\begin{description}
		\item [{\bf (D)}] The homogeneous GLDE \eqref{hGLDE} admits an exponential dichotomy on $\mathbb{R}$ and  {\rm (H)} holds.
	\end{description}
	
	In the sequel, we recall two results about existence of bounded solutions for GLDEs which depend on the exponential dichotomy property.
	
	\begin{prop}{\rm (\cite[Th.4.3]{BFS})}
		\label{hbs}
		Assume that $A\in BV_{loc}(\real,\mathcal{L}(\rn))$ and condition (D) holds. Then, the unique bounded solution of the homogeneous GLDE \eqref{hGLDE} is the null solution. 
	\end{prop}
	
	The next result relies on certain integrability assumptions for the Kurzweil--Stieltjes integral on unbounded intervals. These hypotheses are not necessary in the classical nonautonomous linear framework, because the exponential dichotomy implies that assumptions. For a detailed discussion, we refer to  Remarks~4.9 to 4.11 in \cite{BFS}.  
	
	\begin{prop}{\rm (\cite[Th.4.5]{BFS})}
		\label{nhbs}
		Assume that $A\in BV_{loc}(\real,\mathcal{L}(\rn))$, $f\in G(\real,\rn)$ is bounded, and condition (D) holds. Assume further that the Kurzweil--Stieltjes integrals
		\[
		W_1(t)=\int_{-\infty}^{t}{\rm d}[\U(t)P\U^{-1}(s)]{(f(s)-f(0))}, \text{ and }\; W_2(t)=\int_{t}^{\infty}{\rm d}[\U(t)(I-P)\U^{-1}(s)]{(f(s)-f(0))}
		\]
		exist for all $t\in\real$, and the functions $W_{i}:\real\to\rn$ are bounded, for $i=1,2$. Then the GLDE \eqref{nhGLDE} admits a unique bounded solution. 
	\end{prop}
	
	\begin{remark}\label{rivp}
		To be precise, in Proposition~\ref{nhbs} is proved that the bounded solution of the GLDE \eqref{nhGLDE} is exactly the unique solution of the I.V.P 
		\begin{equation}\label{ivpnh}
		\left\{\begin{array}{l}
		\dfrac{dx}{d\tau}=D[A(t)+f(t)]\\
		\displaystyle x(0)=-\int_{-\infty}^{0}{\rm d}[P\U^{-1}(s)]{(f(s)-f(0))}+\int_{0}^{\infty}{\rm d}[(I-P)\U^{-1}(s)]{(f(s)-f(0))}.
		\end{array}\right.    
		\end{equation}
	\end{remark}
	
	\section{Exponential dichotomy in the $\omega$--periodic case}
	
	It is well known that, in the ODE context,
	the characterisation of the exponential dichotomy for linear $\omega$--periodic systems can be achieved by using the Floquet theory. Nevertheless, to the best of our knowledge, this approach has not been explored in the generalized framework, and will be the aim of this section.
	
	We recall the pivotal result of the Floquet theory for GLDEs, which has been established by Schwabik in \cite{SchwF} and \cite{SCHWABIK3}. 
	\begin{theorem}\label{Floquet}
		Assume that $A\in BV_{loc}(\real,\mathcal{L}(\rn))$, condition {\rm (H)} holds and $A(t+\om)-A(t)=C$, for all $t\in J$, where $\om>0$ and $C\in\mathcal{L}(\rn)$ is a constant matrix. Then there exist a $n\times n$ - matrix valued function $G\colon \real \to \mathcal{L}(\rn)$ which is $\om$-periodic and a constant  matrix $Q\in\mathcal{L}(\rn)$, which is not uniquely determined, such that
		\begin{equation}
		\label{MFF}
		\U(t)=G(t)e^{Qt},\qquad \text{ for all $t\in \real$},
		\end{equation}
		where $Q=\frac{1}{\om}\ln\left(\U(\om)\right)$. 
	\end{theorem}

	Similarly as in the classical Floquet theory, we will say that $\mathcal{U}(\om)$ is the \textit{monodromy matrix}.
	\begin{remark}
		Assuming the hypotheses of Theorem~\ref{Floquet}, we can deduce that for every $t\in\real$ $G(t)$ is invertible, and the function $G^{-1}\colon \real\to \mathcal{L}(\rn)$ is also $\om$--periodic. Indeed, $G(t)G^{-1}(t)=G(t+\om)G^{-1}(t)=I$ leads to $G^{-1}(t)=G^{-1}(t+\om)$. Furthermore, this fact implies the $\om$--byperiodicity of the transition matrix, i.e. $U(t+\om,s+\om)=U(t,s)$ for all $t,s\in\real$. In fact, 
		\begin{equation}
		\label{biper}
		\begin{split}
		U(t+\om,s+\om)&=\U(t+\om)\U^{-1}(s+\om)\\
		&=G(t+\om)e^{Qt}e^{-Qs}G^{-1}(s+\om)\\
		&=G(t)e^{Q(t+\om)}e^{-Q(s+\om)}G^{-1}(s)\\
		&=\U(t)\U^{-1}(s)=U(t,s). 
		\end{split}
		\end{equation}
	\end{remark}

	\begin{lemma}\label{emm}
		Assume the hyphoteses from Theorem~\ref{Floquet}. If $\rho$ is an eigenvalue of 
		the monodromy matrix such that 
		$U(\omega,0)\xi=\rho \xi$, 
		then any solution $x(\cdot)$ of the $\om$--periodic GLDE \eqref{hGLDE} with initial condition $x(0)=\xi$ verifies
		\begin{equation}
		\label{multiplicador}
		x(t+\omega)=\rho x(t), \quad \textnormal{ for all $t\in \mathbb{R}$}.
		\end{equation}
	\end{lemma}
	
	\begin{proof}
		By using the $\omega$--biperiodicity (\ref{biper}), it can be proved inductively that
		\begin{equation}
		\label{MF11}
		U(n\omega,0)\xi=\rho^{n}\xi \quad \textnormal{for any $n\in \mathbb{N}$},
		\end{equation}
		which also implies 
		\begin{equation}
		\label{MF12}
		U(0,n\omega)\xi=\left(\frac{1}{\rho}\right)^{n}\xi \quad \textnormal{for any $n\in \mathbb{N}$}.
		\end{equation}
		
		By using the $\omega$--biperiodicity \eqref{biper} we also can prove that
		\begin{displaymath}
		U(-n\omega,0)\xi=\left(\frac{1}{\rho}\right)^{n}\xi, \quad \textnormal{and} \quad U(0,-n\omega)\xi=\rho^{n}\xi, \quad \textnormal{for any $n\in \mathbb{N}$}.
		\end{displaymath}

		Now, without loss of generality, let us assume that $t\geq0$ with $t\in [(n-1)\omega,n\omega)$. Then, by using the $\omega$--biperiodicity \eqref{biper} combined with the
		identities \eqref{MF11} and \eqref{MF12} we have that
		\begin{displaymath}
		\begin{array}{rcl}
		x(t+\omega)&=& U(t+\omega,0)\xi=U(t+\omega,n\omega)U(n\omega,0)\xi \\
		&=& \rho^{n}U(t,[n-1]\omega)\xi\\
		&=& \rho^{n}U(t,0)U(0,[n-1]\omega)\xi  \\
		&=&\rho^{n}\left(\frac{1}{\rho}\right)^{n-1}U(t,0)\xi= \rho x(t),
		\end{array}
		\end{displaymath}
		and the lemma follows.
		
	\end{proof}
	
	The eigenvalues of the monodromy matrix are also called \textit{Floquet multipliers} in the classi\-cal context and have been studied for $\om$--periodic GLDEs in \cite{Hnilica}. In the next, we will prove that the location of the Floquet multipliers on the complex plane provide a necessary and sufficient condition to determine if the property of exponential dichotomy on $\mathbb{R}$ is verified by the homogeneous GLDE \eqref{hGLDE}.

	\begin{theorem}
		\label{NSC}
		Assume the hyphoteses from Theorem~\ref{Floquet}. The $\om$--periodic GLDE \eqref{hGLDE} admits an {\it exponential dichotomy}
		on $\mathbb{R}$ if and only if the eigenvalues of the monodromy matrix
		are not in the unit circle.
	\end{theorem}
	
	\begin{proof}
		Firstly, we will assume that any eigenvalue $\lambda$ of the monodromy matrix $U(\omega,0)=\U(\omega)$ verifies $|\lambda|\neq 1$. It is easy to see that this implies that any eigenvalue of $Q$ has real part different from zero, which implies that (see \emph{e.g.} \cite[p.430]{Sacker}) the ODE system
		$$
		\dot{y}=Qy
		$$
		has an exponential dichotomy on $\mathbb{R}$ in the classical sense, namely, there exists two constants $K>0$, $\alpha>0$ and a projector $P$ such that
		\begin{displaymath}
		\left\{\begin{array}{rcl}
		\|e^{Qt}Pe^{-Qs}\|\leq Ke^{-\alpha (t-s)}  &\textnormal{if}&  t\geq s \\ \|e^{Qt}(I-P)e^{-Qs}\|\leq Ke^{-\alpha (s-t)}  &\textnormal{if}&  s\geq t.
		\end{array}\right.
		\end{displaymath}
		
		Now, by using Eq. \eqref{MFF} we can see that
		\begin{displaymath}
		\begin{array}{rcl}
		\|\U(t)P\U^{-1}(s)\| &=& \|G(t)e^{Qt}Pe^{-Qs}G^{-1}(s)\| \\
		&\leq& K_{1}e^{-\alpha(t-s)},
		\end{array}
		\end{displaymath}
		for any $t\geq s$, where $K_{1}=K\max\limits_{r\in [0,\omega]}\|G(r)\|\, \max\limits_{r\in [0,\omega]}\|G^{-1}(r)\|$ is well defined since 
		$G$ and its inverse are bounded on $[0,\omega]$.
		
		Similarly, it can be proved that 
		\begin{displaymath}
		\|\U(t)(I-P)\U^{-1}(s)\| \leq K_{1}e^{-\alpha(s-t)},
		\end{displaymath}
		for any $t\leq s$, and the exponential dichotomy of the homogeneous GLDE \eqref{hGLDE} on $\mathbb{R}$ follows.
		
		On the other hand, let us assume that the homogeneous GLDE \eqref{hGLDE} has an exponential dichotomy on $\mathbb{R}$. We will prove that the eigenvalues of the monodromy matrix are not in the unit circle. By contradiction, let us suppose that the monodromy matrix $U(\omega,0)$ associated to the homogeneous GLDE \eqref{hGLDE} has an eigenvalue $\rho$ satisfying $|\rho|=1$. 
		
		Let $\xi$ be an eigenvector such that $U(\omega,0)\xi=\rho \xi$. Using Lemma~\ref{emm} we deduce that the I.V.P
		\begin{equation}\label{ivph+flo}
		\left\{\begin{array}{l}
		\dfrac{dx}{d\tau}=D[A(t)x]\\
		x(0)=\xi
		\end{array}\right.    
		\end{equation}
		has a unique solution defined on $\real$ verifying the Eq. \eqref{multiplicador}. 
		
		Without loss of generality, let us assume that $t\geq0$. Therefore, we can write $t=n\om+s$ with $s\in[0,\om]$, for some $n\in\mathbb{N}\cup\{0\}$. Now, since $U(n\omega,0)\xi=\rho^{n}\xi$ for all $n\in\mathbb{N}\cup\{0\}$, we obtain
		\[
		\|x(t)\|=\|x(n\om+s)\|=|\rho^{n}|\,\|x(s)\|\leq \max\limits_{s\in [-\om,\om]}\|x(s)\|.
		\]
		A similar estimation is deduced considering $t<0$. Thus, we have been obtained a non-zero bounded solution of the homogeneous I.V.P \eqref{ivph+flo}, which contradicts Proposition~\ref{hbs}. 
	\end{proof}

	\section{Revisiting an admissibility result for $\omega$--periodic solutions}
	In the article \cite{BFS}, it was established a result for the existence of a unique $\om$-periodic solution of the inhomogeneous GLDE \eqref{nhGLDE} whenever the functions $A(\cdot)$ and $f(\cdot)$ are $\om$-periodic and the exponential dichotomy property holds. Let us recall this result:
	\begin{prop}{\rm (\cite[Prop.4.8]{BFS})}
		\label{psol}
		Assume the hypotheses from Proposition~\ref{nhbs} and let $C\in\mathcal{L}(\rn)$ be a constant matrix. Suppose that the proyection $P$ is uniquely determined, $A(t+\om)-A(t)=C$, for all $t\in J$, where $\om>0$, and the function $f$ belongs to $G_{\om}(\real,\rn)$. Then the inhomogeneous GLDE \eqref{nhGLDE} admits a unique $\om$-periodic solution.  
	\end{prop}
	\begin{remark}
		As in the previous Remark~\ref{rivp}, the unique $\om$-periodic solution which is obtained by the Proposition~\ref{psol} is indeed the unique solution of the I.V.P \eqref{ivpnh}. A key assumption of Proposition \ref{psol} is that the  $\omega$--periodic GLDE \eqref{hGLDE} has an exponential dichotomy on $\real$. Nevertheless, there are no suggestions how this property could be verified. In this context, Theorem \ref{NSC} provides a necessary and sufficient condition ensuring exponential dichotomy on $\real$, which prompt us to revisit a consequence of Proposition \ref{psol}, namely, an alternative and simpler description of the initial condition of the unique $\omega$--periodic solution.
	\end{remark}

	In the next, we will prove that in the periodic case, the initial condition of the I.V.P \eqref{ivpnh} can be determined by the monodromy matrix $\U(\om)=U(\om,0)$. For this purpose, let us denote the initial condition of the I.V.P \eqref{ivpnh} by
	\[
	x_0:=x(0)=-\int_{-\infty}^{0}{\rm d}[P\U^{-1}(s)]{(f(s)-f(0))}+\int_{0}^{\infty}{\rm d}[(I-P)\U^{-1}(s)]{(f(s)-f(0))}.
	\]
	
	The unique $\om$-periodic solution $x(\cdot,0,x_0)$ of the I.V.P \eqref{ivpnh} satisfies $x(\om,0,x_0)=x(0,0,x_0)=x_0$, and 
	\[
	x(t,0,x_0)=U(t,0)x_0 + f(t)-f(0)-\int_{0}^{t}{\rm d}[U(t,s)](f(s)-f(0)), \qquad \text{for all $t\in\real$}.
	\] 
	It follows by the $\om$-periodicity of $f$ that
	\[
	x(\om,0,x_0)=U(\om,0)x_0-\int_{0}^{\om}{\rm d}[U(\om,s)](f(s)-f(0)).
	\]
	Now, by using the fact that $x(\om,0,x_0)=x_0$ we get 
	\[
	[I-U(\om,0)]x_0=-\int_{0}^{\om}{\rm d}[U(\om,s)](f(s)-f(0)).
	\]
	
	As we are assuming that the $\omega$--periodic GLDE \eqref{hGLDE} has an exponential dichotomy on $\mathbb{R}$, Theorem \ref{NSC} implies that $[I-U(\omega,0)]$
	is an invertible matrix since $\lambda=1$ cannot be an eigenvalue of the monodromy matrix. Hence, we have
	\[
	x_0=-[I-U(\om,0)]^{-1}\int_{0}^{\om}{\rm d}[U(\om,s)](f(s)-f(0)).
	\]
	Thus, we conclude that under hypotheses of Proposition~\ref{psol}, the unique $\om$-periodic solution of the GLDE \eqref{nhGLDE} is actually the unique solution of the I.V.P 
	\begin{equation}\label{ivpp}
	\left\{\begin{array}{l}
	\dfrac{dx}{d\tau}=D[A(t)+f(t)]\\
	\displaystyle x(0)=-[I-U(\om,0)]^{-1}\int_{0}^{\om}{\rm d}[U(\om,s)](f(s)-f(0)).
	\end{array}\right.    
	\end{equation}
	
	\begin{remark}
		The fact that the initial condition $x_0$ is given by the monodromy matrix and the integral involved on it can be  considered on the compact interval $[0,\om]$ also can be deduced by simply calculation of the integrals
		\[
		-\int_{-\infty}^{0}{\rm d}[P\U^{-1}(s)]{(f(s)-f(0))},  \text{ and} \; \int_{0}^{\infty}{\rm d}[(I-P)\U^{-1}(s)]{(f(s)-f(0))}.
		\]
		Indeed, let us denote $\varphi(s)=f(s)-f(0)$ for all $s\in\real$. At first, we pointed out an useful equality. Since the unique $\om$-periodic solution $x(\cdot,0,x_0)$ of the I.V.P \eqref{ivpnh} satisfies $x(-\om,0,x_0)=x(0,0,x_0)=x_0$, proceeding as above, we can prove that
		\begin{equation}\label{minusom}
		[I-U(\om,0)]^{-1}\int_{0}^{\om}{\rm d}[U(\om,s)]\vp(s)=[I-U(-\om,0)]^{-1}\int_{0}^{-\om}{\rm d}[U(-\om,s)]\vp(s).
		\end{equation}
		On the other hand, by using the $\om$-byperiodicity \eqref{biper} of $U(\cdot)$ and the $\om$-periodicity of $\vp$, we have 
		\begin{equation*}
		\begin{split}
		\int_{0}^{\infty}{\rm d}[\U^{-1}(s)]\vp(s)&=\int_{0}^{\om}{\rm d}[\U^{-1}(s)]\vp(s)+ \int_{\om}^{\infty}{\rm d}[\U^{-1}(s)]\vp(s)\\
		&=U(0,\om)\int_{0}^{\om}{\rm d}[U(\om,s)]\vp(s)+ U(0,\om)\int_{0}^{\infty}{\rm d}[U(\om,s+\om)]\vp(s+\om)\\
		&=U(0,\om)\int_{0}^{\om}{\rm d}[U(\om,s)]\vp(s)+ U(0,\om)\int_{0}^{\infty}{\rm d}[\U^{-1}(s)]\vp(s).
		\end{split}
		\end{equation*}
		This implies that 
		\[
		[I-U(\om,0)]\int_{0}^{\infty}{\rm d}[\U^{-1}(s)]\vp(s)=-\int_{0}^{\om}{\rm d}[U(\om,s)]\vp(s).
		\]
		From the exponential dichotomy assumption of the homogeneous $\om$-periodic GLDE \eqref{hGLDE}, it follows  the invertibility of the matrix $[I-U(\om,0)]$ and we get 
		\begin{equation}\label{eqom}
		\int_{0}^{\infty}{\rm d}[(I-P)\U^{-1}(s)]\vp(s)=-(I-P)[I-U(\om,0)]^{-1}\int_{0}^{\om}{\rm d}[U(\om,s)]\vp(s).
		\end{equation}
		In addition, from Eq. \eqref{minusom} and using a similar decomposition as above, we obtain
		\begin{equation}\label{eqmom}
		\begin{split}
		\int_{-\infty}^{0}{\rm d}[P\U^{-1}(s)]\vp(s)&=P[I-U(-\om,0)]^{-1}\int_{0}^{-\om}{\rm d}[U(-\om,s)]\vp(s)\\
		&=P[I-U(\om,0)]^{-1}\int_{0}^{\om}{\rm d}[U(\om,s)]\vp(s).
		\end{split}
		\end{equation}
		Therefore, from Eq. \eqref{eqom} and Eq. \eqref{eqmom} we deduce 
		\begin{equation*}
		\begin{split}
		x_0&=-\int_{-\infty}^{0}{\rm d}[P\U^{-1}(s)]\vp(s)+\int_{0}^{\infty}{\rm d}[(I-P)\U^{-1}(s)]\vp(s)\\
		&=-[I-U(\om,0)]^{-1}\int_{0}^{\om}{\rm d}[U(\om,s)]\vp(s),
		\end{split}
		\end{equation*}
		getting the desired result.
	\end{remark}
	
	We finalize this paper with a comment concerned the hypotheses of the preceding admissi\-bi\-li\-ty result  Proposition~\ref{psol} (\cite[Prop.4.8]{BFS}). In order to obtain the unique $\om$-periodic solution of the inhomogeneous GLDE \eqref{nhGLDE} and as a consequence of Theorem~\ref{NSC}, we notice that gathering the exponential dichotomy assumption (D) and a recent result in \cite[Chapter~9]{BFM}, we are capable to avoid the existence premise of the Kurzweil--Stieltjes integral on unbounded intervals and also that the projection $P$ is uniquely determined.
	
	Let us be more specific, by using Theorem~\ref{NSC} we deduce that the exponential dichotomy on $\real$ implies that the eigenvalues of the monodromy matrix $\U(\om)$ are not in the unit circle, which in turn implies the invertibility of the matrix $[I-U(\om,0)]$. Therefore, the unique $\om$-periodic solution of the $\om$--periodic GLDE \eqref{hGLDE} is the null solution. Hence, by using \cite[Proposition~9.10]{BFM} we get that the inhomogeneous GLDE \eqref{nhGLDE} has a unique $\om$-periodic solution, which is indeed the unique solution of the I.V.P \eqref{ivpp}. In conclusion, we can reformulate the enunciate of Proposition~\ref{psol} as follows:
	\begin{prop}
		Let $C\in\mathcal{L}(\rn)$ be a constant matrix. Assume that $A(t+\om)-A(t)=C$, for all $t\in J$, where $\om>0$, the function $f$ belongs to $G_{\om}(\real,\rn)$, and condition (D) holds. Then the inhomogeneous GLDE \eqref{nhGLDE} admits a unique $\om$-periodic solution. 
	\end{prop}


\begin{thebibliography}{99}
		
		\bibitem{BFM} Ap. Silva, M., et al. (2021). Periodicity. In Bonotto, E. M., et al. (eds.), Generalized Ordinary Differential Equations in Abstract Spaces and Applications, John Wiley $\&$ Sons, Hoboken, pp. 295--316.
		
		
		\bibitem{BFS} Bonotto, E. M., et al. (2018).
		Dichotomies for generalized ordinary differential equations and applications. J. Differential Equations 264, 3131--3173.
		
		\bibitem{BFS2} Bonotto, E. M., et al. (2020).
		Robustness of exponential dichotomies for generalized ordinary differential equations. J. Dynam. Differential Equations 32, 2021--2060.
		
		
		\bibitem{D} Dieudonn\'e, J. (1969). Foundations of Modern Analysis, Academic Press, New York.
		
		

		\bibitem{Hnilica} Hnilica, J. (1976).
		Verallgemeinerte Hill'sche diffeerentialgleichung. \v{C}asopis p\v{e}st. mat. 101, 293--302.
		
		
		
		\bibitem{MST} Monteiro, G. A., et al. (2018). Kurzweil-Stieltjes Integral and its Applications, World Scientific, New Jersey.
		
		
		\bibitem{MT2} Monteiro, G. A., and  Tvrd\'y, M. (2013). Generalized linear differential equations in a Banach space: continuous dependence on a parameter. Discrete Contin. Dyn. Syst. 33, 283--303.
		
		\bibitem{MT1}  Monteiro, G. A., and  Tvrd\'y, M. (2012). On Kurzweil--Stieltjes integral in Banach space, Math. Bohem. 137, 365--381.
		
		
		\bibitem{Sacker} Sacker, R. J., and Sell, G. (1974). 
		Existence of Dichotomies and invariant splittings for linear differential systems I. J. Differential Equations 15, 429--458.
		
		
		
		
		\bibitem{SCHWABIK} Schwabik, \v{S}. (1996). Abstract Perron-Stieltjes integral. Math. Bohem. 121, 425--447.
		
		\bibitem{SchwF} Schwabik, \v{S}. (1973). Floquetova teorie pro zobecn\v{e}n\'e diferenci\'aln\'i rovnice. \v{C}asopis p\v{e}st. mat. 98, 416--418.
		
		\bibitem{SCHWABIK1} Schwabik, \v{S}. (1992). Generalized Ordinary Differential Equations, Ser. Real Anal., Vol. 5, World Scientific, Singapore.
		
		
		\bibitem{SCHWABIK3} Schwabik, \v{S}., et al. (1979). Differential and Integral Equations. Boundary Value Problems and Adjoints, Academia Praha, Reidel Dordrecht.
		
		
		
		
	\end{thebibliography}
\end{document}